\documentclass[12pt]{amsart}
\usepackage{}
\usepackage{amssymb}
\usepackage{amsmath}
\usepackage{cases}
\usepackage{mathrsfs}
\usepackage{amsfonts}
\usepackage{txfonts}
\usepackage{pifont}
\usepackage{cite}
\usepackage{cleveref}
\usepackage{color}

\newtheorem{theorem}{Theorem}[section]
\newtheorem{lemma}[theorem]{Lemma}
\newtheorem{corollary}[theorem]{Corollary}

\theoremstyle{definition}

\newtheorem{example}[theorem]{Example}

\theoremstyle{remark}
\newtheorem{remark}[theorem]{Remark}

\newcommand{\ind}{\textup{ind}}
\newcommand{\rank}{\textup{rank}}

\numberwithin{equation}{section}

\def\sbmatrix{\left[\begin{array}}
\def\endsbmatrix{\end{array}\right]}

\begin{document}

\title{On the existence of group inverses of Peirce corner matrices}
\author{Daochang Zhang}
\address{College of Sciences, Northeast Electric Power University, Jilin, P.R. China.
}
\email{daochangzhang@126.com}
\thanks{The first author is supported by the National Natural Science Foundation of China (NSFC) (No. 61672149; No. 51708091). 
The second author is supported by the Ministry of Education, Science and Technological Development, Republic of Serbia (No. 174007).}
\author{Dijana Mosi\' c}
\address{Faculty of Sciences and Mathematics, University of Ni\v s, P.O.
Box 224, 18000 Ni\v s, Serbia.} \email{dijana@pmf.ni.ac.rs}
\author{Tin-Yau Tam}
\address{Department of Mathematics and Statistics, University of Nevada, Reno, NV 89557, USA.}
\email{ttam@unr.edu}


\date{}



\begin{abstract}
We give some statements that are equivalent to the existence of group inverses of Peirce corner matrices of a $2 \times 2$ block matrix and its generalized Schur complements.
As applications, several new results for the Drazin inverses of the generalized Schur complements and the $2 \times 2$ block matrix are obtained and some of them generalize several results in the literature.
\end{abstract}

\maketitle

\section{introduction}
The classical Sherman-Morrison-Woodbury formula
$$
(A-CD^{-1}B)^{-1} = A^{-1} + A^{-1}C(D-BA^{-1}C)^{-1}BA^{-1}
$$
expresses the inverse of the
 Schur complement $A-CD^{-1}B$ of $D$ of the $2\times 2$ block matrix
 $$
 M=\sbmatrix{cc}
A&C\\B&D\endsbmatrix
 $$
   in terms of
the inverses of $A$ and its Schur complement
$D-BA^{-1}C$  of $M$, where $A$ and $D$ are invertible matrices, but not
necessarily with the same size, and $B$ and $C$ are matrices with
 appropriate sizes such that $D-BA^{-1}C$ (and so $A-CD^{-1}B)$
is invertible \cite{Sherman(1950),Woodbury(1950)}. There are
 applications of inverse matrix  formulae
of such type in various fields such as statistics, optimization, networks,
numerical analysis, structural analysis,  partial different equations
\cite{HagerUpdating(1989)}. Formulae
of such type have been developed in the context of generalized
inverses, such as the Moore-Penrose inverse
\cite{bak2003,Meyer(1973)}, the weighted Moore-Penrose inverse
\cite{wei2001}, the group inverse \cite{cas-gon2013}, the weighted
Drazin inverse \cite{chen2008amc}, the generalized Drazin inverse
\cite{deng2011}, and  the Drazin inverse \cite{FIL2015}.

Researchers have tried to find
formulae for the Drazin inverse of $M$ in terms
of its blocks, where $A$ and $D$ are
square matrices \cite{HLW,Hartwig(1977),Meyer(1977)}.
This problem originates from the singular system of differential equations \cite{CMR1976,CampbellLAMA1983},
and was proposed by Campbell and Meyer \cite{Campbell1991}, but it is
still an open problem if no additional assumptions are made on the
blocks.

There are also some significant definitions in the ring theory.
Let $R$ be a ring.
A ring $S\subseteq R$ (with the same multiplication as R,
but not assumed to have an identity initially) is said to be a {\it corner ring} (or simply a {\it corner})
of $R$ if there exists an additive subgroup $C\subseteq R$ such that
\begin{equation}\label{defcorner}
 R=S\oplus C,
 \quad S\cdot C\subseteq C,\quad \text{and}\quad C\cdot S\subseteq C.
\end{equation}
In this case, we write $S\prec R$ , and we call any subgroup $C$ satisfying \eqref{defcorner}
a {\it complement} of the corner ring $S$ in $R$.
In general, such a complement $C$ is far from being unique.
If a corner $S$ of a ring $R$ happens to have a unique complement,
we shall call $S$ a {\it rigid corner} of $R$, and write $S\prec_r R$.

In 2006,  Lam \cite{Lam2006} proved that a corner ring of any ring $R$ must have
an identity, although this may not be the identity of $R$.

\begin{remark}\cite[Proposition 2.2]{Lam2006}
Let $S\prec R$, with a complement $C$. If $1=e+f$ for some $e\in S$ and $f\in C$,
then $e$ is an identity of the ring $S$.
In particular, the decomposition $1=e+f$ is independent of the choice of the complement $C$, and $e, f$ are
complementary idempotents in $R$.
\end{remark}

Let $e, f$ be complementary idempotents in a ring $R$. Then
$R_e$ is called the {\it Peirce corner} of $R$ (arising from the idempotent $e$) such that
\begin{enumerate}
\item $R_e := eRe \prec R$, which is the largest subring (resp. corner) of $R$ having $e$ as identity element.  \\
\item $R_e\prec_r R$ (i.e., $R_e$ is rigid in $R$), with a unique complement
$$C_e := fRe \oplus eRf \oplus fRf = \{r \in R : ere = 0\},$$
\end{enumerate}
where $C_e$ is called the Peirce complement of $R_e$.

Recall that, for any ring $R$ (with identity), Jacobson's Lemma states that if $1-ab$ is invertible, then so is
$1-ba$ and
$$(1-ba)^{-1}=1+b(1-ab)^{-1}a.$$
The {\it group inverse} of a complex square matrix $A$ is the unique matrix $A^{\#}$ such that
\begin{equation*}\label{defin1}
  AA^{\#}A=A,~~~ A^{\#}AA^{\#}=A^{\#},~~~ AA^{\#}=A^{\#}A.
\end{equation*}
The {\it Drazin inverse} of  $A$ is the unique matrix $A^d$ such that
\begin{equation*}\label{defin1}
  AA^d=A^dA,~~~ A^dAA^d=A^d,~~~ A^{k}=A^{k+1}A^d,
\end{equation*}
 where $k$ is the smallest non-negative integer such that $\rank(A^{k}) = \rank(A^{k+1})$, called index of $A$ and denoted by $\ind(A)$.
We also denote $A^e=AA^d$ and $A^\pi=I-A^e$. If $\ind(A)=1$, then $A^{d}=A^{\#}$.

%

In this paper, our aim is, by the Peirce corner theory,
to establish a relationship between group inverses of a $2 \times 2$ block matrix and its generalized Schur complements,
and to find statements that are equivalent to the existence of group inverses of Peirce corner matrices of a $2 \times 2$ block matrix and its generalized Schur complements. Utilizing the equivalent statements, we derive some new results about Drazin inverses of the generalized Schur complements and the $2 \times 2$ block matrix, and generalize several results in the literature.

Throughout this paper,
for notational convenience, we denote
$$S=A-CD^dB, \quad s=A^eSA^e,\quad Z=D-BA^dC,\quad z=D^eZD^e,$$
$$M=\begin{bmatrix}
    A & C \\
    B & D
 \end{bmatrix},
\quad E=\begin{bmatrix}
    A^{e} & 0 \\
 0 & D^{e}
 \end{bmatrix},
 \quad M_E=EME=\begin{bmatrix}
    AA^{e} & A^{e}CD^{e} \\
             D^{e}BA^{e} & DD^{e}
 \end{bmatrix},$$
where $A\in\mathbb{C}^{n\times n}$, $B\in\mathbb{C}^{m\times n}$,
$C\in\mathbb{C}^{n\times m}$, and $D\in\mathbb{C}^{m\times m}$.
Here $\mathbb{C}^{m\times n}$ is the set of $m\times n$ complex
matrices.
We treat $\sum_{i=m}^n \ast=0$ whenever $m<n$, which is used in Section 3 and
Section 4. We denote by $I$ the identity matrix of proper size.

\section{Existences of group inverses}
In this section, we give statements that are equivalent to the existence of group
inverses of Peirce corner matrices of a $2 \times 2$
block matrix and its generalized Schur complements.

Here we cite a result which characterizes matrices with the same eigenprojection.
\begin{lemma}{\rm \cite[Theorem 2.1]{CGK2000}}\label{same eigenprojection}
  Let $A,B \in \mathbb{C}^{n\times n}$ be Drazin invertible. Then the following conditions are equivalent:
\begin{enumerate}
\item $A^\pi=B^\pi$;
\item $A^\pi B=BA^\pi,BA^\pi$ is nilpotent and $B+A^\pi$ is nonsingular;
\item $I+A^d(B-A)$ is nonsingular, $A^\pi B=BA^\pi$ and $BA^\pi$ is nilpotent;
\item $B^d=(I+A^d(B-A))^{-1}A^d$;
\item $B^d-A^d=A^d(A-B)B^d$.
\end{enumerate}
\end{lemma}
Now we can give our first main result.
\begin{theorem}\label{equivalent}
The following statements are equivalent:
\begin{enumerate}
  \item $M_E^\#$ exists, and $M_E^e=E$;
   \item  $s^\#$ exists, and $s^e=A^e$;
    \item $z^\#$ exists, and $z^e=D^e$;
    \item $A^\pi+s$ is invertible;
    \item $I+A^d(s-A)$ is invertible (or $I+A^d(S-I)A$ is invertible);
    \item $s^d=(I+A^d(s-A))^{-1}A^d$;
     \item $D^\pi+z$ is invertible;
    \item $I+D^d(z-D)$ is invertible (or $I+D^d(Z-I)D$ is invertible);
    \item $z^d=(I+D^d(z-D))^{-1}D^d$;
    \item $s^\#, z^\#$ exist, $s^\#=A^dCz^\#BA^d+A^d$, and $A^dCz^\#=s^\#CD^d$;
    \item $s^\#, z^\#$ exist, $z^\#=D^dBs^\#CD^d+D^d$, and $D^dBs^\#=z^\#BA^d$;
    \item $M_E^\#$ exists, and
$$M_E^\#=
\begin{bmatrix}
   A^d+A^d C z^\#  B A^d&-A^d C z^\# \\-z^\# B A^d& z^\#
\end{bmatrix}=
\begin{bmatrix}
s^\#& -s^\# C D^d\\- D^d B s^\# & D^d B s^\#  C D^d+D^d
\end{bmatrix}.$$
\end{enumerate}
 \end{theorem}
\begin{proof}
We first show the implication (1)$\Rightarrow$(3) holds.
Note that we have the following two representations for the Peirce corner matrix $M_E$ of $M$.
\begin{align}\label{Msz1}
  \begin{bmatrix}
     AA^{e}  & A^{e}CD^{e} \\
     D^{e}BA^{e}  & DD^{e}
   \end{bmatrix}
 =
  \begin{bmatrix}
     A^{e}        & 0 \\
     D^{e}BA^{d} & D^{e}
   \end{bmatrix}
  \begin{bmatrix}
     AA^{e} & 0 \\
     0      & DD^{e}-D^{e}BA^{d}CD^{e}
   \end{bmatrix}
  \begin{bmatrix}
     A^{e} & A^{d}CD^{e} \\
     0     & D^{e}
   \end{bmatrix},
\end{align}
\begin{align}\label{Msz2}
 \begin{bmatrix}
     AA^{e}  & A^{e}CD^{e} \\
     D^{e}BA^{e}  & DD^{e}
   \end{bmatrix}
 =
  \begin{bmatrix}
     A^{e} & A^{e}CD^{d} \\
     0     & D^{e}
   \end{bmatrix}
  \begin{bmatrix}
     AA^{e}-A^{e}CD^{d}BA^{e} & 0 \\
     0                        & DD^{e}
   \end{bmatrix}
   \begin{bmatrix}
     A^{e}        & 0 \\
     D^{d}BA^{e} & D^{e}
   \end{bmatrix}.
\end{align}
Let $\mathbb{W}=\{M\in \mathbb{C}^{n\times n}\}$ and $\mathbb{W}_E=\{M_E\in \mathbb{C}^{n\times n}\}$.
We observe that all of the matrices above are contained in the Peirce corner $\mathbb{W}_E$ of $\mathbb{W}$ with identity  $E$.
Since $\mathbb{W}_E$ is a finite dimensional algebra over $\mathbb{C}$ with identity $E$, by $EM_E=M_EE=M_E$, we conclude that $\mathbb{W}_E$ is Dedekind finite (see \cite[Corollary 21.27]{Lam1991}), i.e., if $XY=E$ then $YX=E$ for the arbitrary $X,Y\in \mathbb{W}_E$.
Since
\begin{equation*}
\begin{bmatrix}
     A^{e}        & 0 \\
     -D^{e}BA^{d} & D^{e}
   \end{bmatrix}
  M_E
  \begin{bmatrix}
     A^{e} & -A^{d}CD^{e} \\
     0     & D^{e}
   \end{bmatrix}
=\begin{bmatrix}
     AA^{e} & 0 \\
     0     & D^{e}ZD^{e}
   \end{bmatrix}
\end{equation*}
and
\begin{equation*}
  \begin{bmatrix}
     AA^{e} & 0 \\
     0     & D^{e}ZD^{e}
   \end{bmatrix}
\times
\begin{bmatrix}
     A^{e} & A^{d}CD^{e} \\
     0     & D^{e}
   \end{bmatrix}
M_E^{\#}
\begin{bmatrix}
     A^{e}        & 0 \\
     D^{e}BA^{d} & D^{e}
   \end{bmatrix}
=E,
\end{equation*}
we conclude that
$\begin{bmatrix}
     AA^{e} & 0 \\
     0     & D^{e}ZD^{e}
   \end{bmatrix}^{\#}$ exists, and
$\begin{bmatrix}
     AA^{e} & 0 \\
     0     & D^{e}ZD^{e}
   \end{bmatrix}^e=E$.
Then, by the uniqueness of the group inverse, $z^\#$ exists and $z^e=D^e$.

Similarly, it can be proved  that  (1), (2), and (3) are equivalent.

To prove that the implications  (1)$\Rightarrow$(10) and (1)$\Rightarrow$(11) hold,  we combine the equivalence of (1), (2), and (3) to give

\begin{align*}
  &\begin{bmatrix}
     A^{e} & -A^{d}CD^{e} \\
     0     & D^{e}
   \end{bmatrix}\
   \begin{bmatrix}
     A^{d} & 0 \\
     0      &z^\#
   \end{bmatrix}
   \begin{bmatrix}
     A^{e}        & 0 \\
     -D^{e}BA^{d} & D^{e}
   \end{bmatrix}\\
=
  & \begin{bmatrix}
     A^{e}        & 0 \\
     -D^{d}BA^{e} & D^{e}
   \end{bmatrix}
  \begin{bmatrix}
     s^\# & 0 \\
     0                        & D^{d}
   \end{bmatrix}
\begin{bmatrix}
     A^{e} & -A^{e}CD^{d} \\
     0     & D^{e}
   \end{bmatrix},
\end{align*}
i.e.,
$$
\begin{bmatrix}
   A^d+A^d C z^\#  B A^d&-A^d C z^\# \\-z^\# B A^d& z^\#
\end{bmatrix}=
\begin{bmatrix}
s^\#& -s^\# C D^d\\- D^d B s^\# & D^d B s^\#  C D^d+D^d
\end{bmatrix},
$$
and so
\begin{gather*}
   s^\#=A^d+A^d C z^\#  B A^d,\qquad
   s^\# C D^d=A^d C z^\#,  \\
  D^d B s^\#= z^\# B A^d, \qquad
  z^\#=D^d B s^\# C D^d+D^d.
\end{gather*}

The implication (10)$\Rightarrow$(2) holds because $$s^\#=A^dCz^\#BA^d+A^d=s^\#CD^dBA^d+A^d=s^\#(A-S)A^d+A^d=s^\#-s^\#SA^d+A^d$$
implies $A^d=s^\#SA^d$. Hence, $A^e=A^dA=s^\#SA^dA=s^\#A^eSA^e=s^\#s=s^e$.

In a similar way, we can verify that (11)$\Rightarrow$(3). Thus, (1), (2), (3), (10), and (11) are equivalent.

By \eqref{Msz1}, \eqref{Msz2}, and simple calculation, we obtain the representation of $M^\#_E$ in (12)
and the equivalence between (12), and (1)--(3) holds.

The rest can be obtained by Lemma \ref{same eigenprojection} and Jacobson's Lemma.
%
%
%
\end{proof}

The equivalent statements in the above theorem distinctly provide  relationship between the existence of group inverses of Peirce corner matrices
of a $2\times 2$ block matrix and its generalized Schur complements.
Although some extra assumptions are imposed,
the theorem yields a new method in the research field
and can combine the two different research directions of the generalized inverses of the generalized Schur complements and the $2\times 2$ block matrix. The result
 is applied to Section 3 and Section 4,
and generalizes a few results in the literature.

We next conclude with a remark to provide for the application in Section 3.

\begin{remark}\label{rfilgroupinverse}
We note that \cite[Lemma 2.4]{FIL2015} states if $A^dCD^\pi
Z^dBA^d=\linebreak A^dCD^dZ^\pi BA^d,$
then $s^\#$ exists and $s^\#=A^d+A^dCZ^dBA^d$, and so $s^e=A^e$.
In this case, by Theorem \ref{equivalent},
$s^\#=A^d+A^dCz^\#BA^d$. Since the group inverse of $s$ is unique,
we can obtain if $A^dCD^\pi Z^dBA^d=A^dCD^dZ^\pi BA^d,$ then
$$A^dCZ^dBA^d=A^dCz^\#BA^d.$$ We can also conclude that if $D^dBA^\pi
S^dCD^d=D^dBA^dS^\pi CD^d,$
then $z^\#$ exists, $z^\#=D^d+D^dBS^dCD^d$ and $z^e=D^e$.
Similarly, if $D^dBA^\pi S^dCD^d=D^dBA^dS^\pi CD^d$, then
$$D^dBS^dCD^d=D^dBs^\#CD^d.$$
\end{remark}

Applying Jacobson's Lemma, we can obtain equivalent statements that are similar to those in Theorem \ref{equivalent}.
Furthermore, similar to Theorem \ref{equivalent}, by \cite[Theorem 2.3]{KS} or \cite[Lemma 1.1]{CGK2000}, we can verify a remark as follows.
\begin{remark}
(a) Let $Z=D-BA^dC$, $z_1=D^\pi ZD^\pi$,
 \begin{multline*}
~~M=\begin{bmatrix}
    A & C \\
    B & D
 \end{bmatrix},
 ~~G=\begin{bmatrix}
    A^{e} & 0 \\
 0 & D^\pi
 \end{bmatrix},
 ~~M_G=GMG=\begin{bmatrix}
    AA^{e} & A^{e}CD^\pi \\
             D^\pi BA^{e} & DD^\pi
 \end{bmatrix},
\end{multline*}
where $A\in\mathbb{C}^{n\times n}$, $B\in\mathbb{C}^{m\times n}$,
$C\in\mathbb{C}^{n\times m}$, and $D\in\mathbb{C}^{m\times m}$.
The following statements are equivalent:
\begin{enumerate}
  \item $M_G^\#$ exists, and $M_G^e=G$;
    \item $z_1^\#$ exists, and $z_1^e=D^\pi$;
     \item $D^e+z_1$ is invertible;
    \item $I+D^\pi(z_1-I)$ is invertible (or $I+D^\pi(Z-I)$ is invertible);
    \item $M_G^\#$ exists, and
$$M_G^\#=
\begin{bmatrix}
   A^d+A^d C z_1^\#  B A^d&-A^d C z_1^\# \\-z_1^\# B A^d& z_1^\#
\end{bmatrix}.
$$
\end{enumerate}

(b) Let $S=A-CD^dB$, $s_1=A^\pi SA^\pi$,
 \begin{multline*}
 ~~M=\begin{bmatrix}
    A & C \\
    B & D
 \end{bmatrix},
 ~~H=\begin{bmatrix}
    A^\pi & 0 \\
 0 & D^e
 \end{bmatrix},
 ~~M_H=HMH=\begin{bmatrix}
    AA^\pi & A^\pi CD^e \\
             D^e BA^\pi & DD^e
 \end{bmatrix},
\end{multline*}
where $A\in\mathbb{C}^{n\times n}$, $B\in\mathbb{C}^{m\times n}$,
$C\in\mathbb{C}^{n\times m}$, $D\in\mathbb{C}^{m\times m}$.
The following statements are equivalent:
\begin{enumerate}
  \item $M_H^\#$ exists, and $M_H^e=H$;
   \item  $s_1^\#$ exists, and $s_1^e=A^\pi$;
    \item $A^e+s_1$ is invertible;
    \item $I+A^\pi(s_1-I)$ is invertible (or $I+A^\pi(S-I)$ is invertible);
    \item $M_H^\#$ exists, and
$$M_H^\#=
\begin{bmatrix}
s_1^\#& -s_1^\# C D^d\\- D^d B s_1^\# & D^d B s_1^\#  C D^d+D^d
\end{bmatrix}.
$$
\end{enumerate}
\end{remark}

By direct calculation, we can obtain the following corollary.
\begin{corollary} \label{corszM}

(a) The following statements are equivalent:
\begin{enumerate}
\item $M_E^\#$ exists, and
$$M_E^\#=
\begin{bmatrix}
   A^d+A^d C z^d  B A^d&-A^d C z^d \\-z^d B A^d& z^d
\end{bmatrix};
$$
\item $z^\#$ exists, $A^dCz^\pi D^e=0$ and $D^ez^\pi BA^d=0$;
\item $z^\#$ exists, $A^dCz^e=A^dCD^e$ and $D^eBA^d=z^eBA^d$.
\end{enumerate}
(b) The following statements are equivalent:
\begin{enumerate}
\item $M_E^\#$ exists, and
$$M_E^\#=
\begin{bmatrix}
s^d& -s^d C D^d\\- D^d B s^d & D^d B s^d  C D^d+D^d
\end{bmatrix};
$$

\item $s^\#$ exists, $A^es^\pi CD^d=0$ and $D^dBs^\pi A^e=0$;
\item $s^\#$ exists, $A^eCD^d=s^eCD^d$ and $D^dBs^e=D^dBA^e$.
\end{enumerate}
\end{corollary}

We  note that $M_E^e=\begin{bmatrix}
  A^e&0\\0&z^e
\end{bmatrix}$ in Corollary \ref{corszM}(a).
If we assume $A^dCz^e=A^dCD^e$ and $D^eBA^d=z^eBA^d$, then
$A^dCD^\pi z^d BA^d=0=A^dCD^d z^\pi BA^d.$

In the following result, we present necessary and sufficient
conditions for the solvability of a matrix equation, and give the the expression of the general solution if the system is solvable.

\begin{theorem}
  Let $A\in\mathbb{C}^{n\times n}$, $B\in\mathbb{C}^{m\times n}$,
$C\in\mathbb{C}^{n\times m}$, and $D\in\mathbb{C}^{m\times m}$. The
system
  \begin{equation}\label{eq(A^{e},0)}
    \begin{bmatrix}
      A & C \\
      B & D
    \end{bmatrix}
 \begin{bmatrix}
    A^{e}X \\
    Z^{e}Y
  \end{bmatrix}
 =
  \begin{bmatrix}
    A^{e} \\
    0
  \end{bmatrix}
  \end{equation}
has a solution if and only if $A^\pi CZ^dBA^d=0$ and $Z^\pi
BA^d=0$. In this case, the general solution is
$$X=A^{d}+A^{d}CZ^{d}BA^{d}+A^\pi U \quad and\quad
    Y=-Z^{d}BA^{d}+Z^\pi V,$$ for arbitrary
$U\in\mathbb{C}^{n\times n}$ and $V\in\mathbb{C}^{m\times n}$.
\end{theorem}
\begin{proof} If the system $\eqref{eq(A^{e},0)}$ has a solution, then
\begin{equation}\label{eqA^{e}}
  AA^{e}X+CZ^{e}Y=A^{e},
\end{equation}
and
\begin{equation}\label{eq0}
  BA^{e}X+DZ^{e}Y=0.
\end{equation}
By \eqref{eqA^{e}}, we have
\begin{equation}\label{A^{e}X}
  A^{e}X=A^{d}-A^{d}CZ^{e}Y
\end{equation}
which is substituted into \eqref{eq0} to have
$BA^{d}-BA^{d}CZ^{e}Y+DZ^{e}Y=0$, i.e.,
\begin{equation*}
  -BA^{d}=(D-BA^{d}C)Z^{e}Y=ZZ^{e}Y.
\end{equation*}
Then
\begin{equation}\label{ZeYCC}
  Z^{e}Y=-Z^{d}BA^{d},
\end{equation}
which is substituted into \eqref{A^{e}X} to have
\begin{equation}\label{AeXCC}
  A^{e}X=A^{d}+A^{d}CZ^{d}BA^{d}.
\end{equation} Therefore,
\begin{eqnarray*}
A^e&=&A(A^{e}X)+C(Z^{e}Y)=AA^{d}+AA^{d}CZ^{d}BA^{d}-CZ^{d}BA^{d}\\&=&A^e-A^\pi CZ^{d}BA^{d}
\end{eqnarray*} gives $A^\pi CZ^{d}BA^{d}=0$. We obtain
\begin{eqnarray*}0&=&B(A^{e}X)+D(Z^{e}Y)=BA^{d}+BA^{d}CZ^{d}BA^{d}-DZ^{d}BA^{d}
\\&=&BA^{d}-ZZ^{d}BA^{d}=Z^\pi BA^{d}.
\end{eqnarray*}
Following \cite[p. 52]{BIG}, we get the general solution of
$A^{e}X=A^{d}+A^{d}CZ^{d}BA^{d}$ as
$X=A^{d}+A^{d}CZ^{d}BA^{d}+A^\pi U$, for arbitrary
$U\in\mathbb{C}^{n\times n}$. Similarly, the general solution of the matrix
equation $Z^{e}Y=-Z^{d}BA^{d}$ is $Y=-Z^{d}BA^{d}+Z^\pi V$, for
arbitrary $V\in\mathbb{C}^{m\times n}$.

Suppose that $A^\pi CZ^dBA^d=0$ and $Z^\pi BA^d=0$. Set
$X=A^{d}+A^{d}CZ^{d}BA^{d}+A^\pi U$ and $Y=-Z^{d}BA^{d}+Z^\pi V$,
 for arbitrary
$U\in\mathbb{C}^{n\times n}$ and $V\in\mathbb{C}^{m\times n}$. We
can easily verify that $AA^{e}X+CZ^{e}Y=A^{e}$ and
$BA^{e}X+DZ^{e}Y=0$, i.e., the system $\eqref{eq(A^{e},0)}$ has
a solution.
\end{proof}

We next consider the group inverse of $s$ under the above equivalence restriction.
\begin{corollary}
Let $A\in\mathbb{C}^{n\times n}$, $B\in\mathbb{C}^{m\times n}$,
$C\in\mathbb{C}^{n\times m}$, $D\in\mathbb{C}^{m\times m}$, and
let $A^dCD^\pi Z^dBA^d=0$. If the system \eqref{eq(A^{e},0)} has a
solution, then $s^\#$ exists,
$$s^\#=A^{d}+A^{d}CZ^{d}BA^{d},$$
 and $s^e=A^e.$
\end{corollary}
\begin{proof}
Note that $\eqref{eqA^{e}}-CD^{d}\times \eqref{eq0}$ yields
\begin{equation*}
  A^{e}=AA^{e}X+CZ^{e}Y-CD^{d}BA^{e}X-CD^{e}Z^{e}Y.
\end{equation*}
Then
\begin{equation*}
  A^{e}=AA^{e}X-A^{e}CD^{d}BA^{e}X+A^{e}CD^{\pi}Z^{e}Y,
\end{equation*}
which implies, by \eqref{ZeYCC} and $A^dCD^\pi Z^dBA^d=0$,
\begin{equation*}
  A^{e}=AA^{e}X-A^{e}CD^{d}BA^{e}X=sA^{e}X.
\end{equation*}
By \eqref{AeXCC},
$$A^{e}=s(A^{d}+A^{d}CZ^{d}BA^{d}).$$
Since $\mathbb{W}_{A^e}=\{A^eM^\prime A^e\in \mathbb{C}^{n\times
n}:M^\prime\in \mathbb{C}^{n\times n}\}$ is a finite dimensional
algebra over $\mathbb{C}$ with identity $A^e$, by
$A^eM_{A^e}=M_{A^e}A^e=M_{A^e}$, we conclude that $\mathbb{W}_{A^e}$ is
also Dedekind finite as in the proof of Theorem \ref{equivalent}.
Then
$$A^{e}=(A^{d}+A^{d}CZ^{d}BA^{d})s.$$
Thus $s^\#$ exists,
$$s^\#=A^{d}+A^{d}CZ^{d}BA^{d},$$
 and $s^e=A^e.$
\end{proof}

\section{applications to $S^d$ and $Z^d$}
Applying Theorem \ref{equivalent}, we obtain new results for the
Drazin inverses of the generalized Schur complements $S$ and $Z$ in this
section. As a consequence,
we generalize several results in the literature, and
recover some generalized Sherman-Morrison-Woodbury formulae.

We first need a existing lemma. In what follows, recall that $s=AA^dSAA^d$ and $z=DD^dZDD^d$.
\begin{lemma}\label{(lemma) P nilpotent}{\rm \cite{FIL2015}}
If $A^{\pi}CD^{d}B=0$, then $s=SAA^{d}$ and
$$
  S^{d}=s^{d}+\sum_{i=0}^{k-1}(s^{d})^{i+2}SA^{i}A^{\pi},
$$
where
 $k= \ind(A)$.
\end{lemma}

By Lemma \ref{(lemma) P nilpotent} and Theorem \ref{equivalent},
we get the following result.

\begin{theorem}\label{(th)Sdbysg}
If any of (1)--(12) in Theorem \ref{equivalent} is valid, and $A^{\pi}CD^{d}B=0$, then
$$
  S^{d}=s^\#+\sum_{i=0}^{k-1}(s^\#)^{i+2}SA^{i}A^{\pi},
$$
where
$s^\#=A^dCz^\#BA^d+A^d$, and
 $k= \ind(A)$.
\end{theorem}

The following conclusion follows from Remark
\ref{rfilgroupinverse} and Theorem \ref{(th)Sdbysg}, because
$A^dCD^\pi Z^dBA^d=A^dCD^dZ^\pi BA^d$ implies that $s^\#$ exists
and $s^e=A^e$, which belongs in one of the equivalent statements
in Theorem \ref{equivalent}, and also $A^dCZ^dBA^d=A^dCz^\#BA^d.$

\begin{corollary}{\rm \cite[Theorem 2.5]{FIL2015}}\label{KHFILTH2.5}
  If $A^{\pi}CD^{d}B=0$ and $A^dCD^\pi Z^dBA^d=A^dCD^dZ^\pi BA^d$, then
$$
  S^{d}=A^{d}+A^{d}CZ^{d}BA^{d}+\sum_{i=0}^{k-1}(A^{d}+A^{d}CZ^{d}BA^{d})^{i+2}SA^{i}A^{\pi},
$$
or alternatively
\begin{multline*}
S^{d}=A^{d}+A^{d}CZ^{d}BA^{d}-\sum_{i=0}^{k-1}(A^{d}+A^{d}CZ^{d}BA^{d})^{i+1}A^{d}CZ^{d}BA^{i}A^{\pi}\\
 +\sum_{i=0}^{k-1}(A^{d}+A^{d}CZ^{d}BA^{d})^{i+1}A^{d}C(Z^{d}D^{\pi}-Z^{\pi}D^{d})BA^{i},
\end{multline*}
where
$k=\ind(A)$.
\end{corollary}

We can see from \cite{FIL2015} how Corollary \ref{KHFILTH2.5} gives
and generalizes the Sherman-Morrison-Woodbury formula and some results
in \cite{Dijana2013,Dopazo2013,Shakoor2013,Wei2002}.

The following corollary, which is a dual version of Theorem \ref{(th)Sdbysg}, can be proved similarly.
\begin{corollary}\label{(COR)Qnilpotent}
 If any of (1)--(12) in Theorem \ref{equivalent} is valid, and $CD^{d}BA^{\pi}=0$, then $s=AA^dS$, and
\begin{align*}
    S^{d}=s^\#+\sum_{i=0}^{k-1}A^{\pi}A^{i}S(s^\#)^{i+2},
\end{align*}
where
$s^\#=A^dCz^\#BA^d+A^d$, and
 $k= \ind(A)$.
\end{corollary}
We would like to point out that Corollary \ref{(COR)Qnilpotent} generalizes \cite[Theorem 3]{Dijana2013}, \cite[Theorem 2.2]{Dopazo2013}, \cite[Theorem 2.2]{Shakoor2013}, and \cite[Theorem 2.1]{Wei2002}.

Similarly,  we can obtain the following result.
\begin{corollary}\label{(CORO)Dpinilpotent}
If  any of (1)--(12) in Theorem \ref{equivalent} is valid, and $D^\pi BA^dC=0$, then $z=ZDD^{d}$ and
$$
  Z^{d}=z^\#+\sum_{i=0}^{k-1}(z^\#)^{i+2}ZD^{i}D^{\pi},
$$
where
$z^\#=D^dBs^\#CD^d+D^d$, and
 $k= \ind(D)$.
\end{corollary}

\section{applications to $M^d$}
Under new conditons, using Theorem \ref{equivalent}, we present
several representations for the Drazin inverse of a $2\times 2$
block matrix $M$ in terms of its blocks and also elements $z$ and
$s$ which include corresponding generalized Schur complements.

We cite one result as follows.
\begin{lemma}\label{P+Q}\rm \cite[Theorem 2.1]{Hartwig(2001)}
Let $P$ and $Q\in \mathbb{C}^{n\times n}$. If $PQ=0$, then
$$
  (P+Q)^{d}=Q^{\pi}\sum_{i=0}^{t-1}Q^{i}(P^{d})^{i+1}+\sum_{i=0}^{s-1}(Q^{d})^{i+1}P^{i}P^{\pi},
$$
where $s= \ind(P)$ and $t=\ind(Q)$.
\end{lemma}

In the case that one of equivalent statements of Theorem
\ref{equivalent} holds, $ACD^\pi=0$ and $DBA^\pi=0$, we obtain the
following representation for the Drazin inverse of $M$.

\begin{theorem}\label{applMd-1}
 If any of (1)--(12) of Theorem \ref{equivalent} is valid, $ACD^\pi=0$, and $DBA^\pi=0$, then
\begin{eqnarray}M^d&=&\left(\begin{bmatrix}
    (CBA^\pi)^\pi & 0 \\
             0 & (BCD^\pi)^\pi
 \end{bmatrix}\right. \nonumber\\
 &-&\left.\sum_{m=1}^{t-1}\begin{bmatrix}
             0 & (CBA^\pi)^dCD^\pi \\
             BA^\pi (CBA^\pi)^d& 0\end{bmatrix}^m\begin{bmatrix}
    A^mA^\pi & 0 \\
             0 & D^mD^\pi
 \end{bmatrix}\right) \nonumber\\
&\times&\sum_{i=0}^{r-1}\begin{bmatrix}
    AA^\pi & CD^\pi \\
             BA^\pi & DD^\pi
 \end{bmatrix}^{i}\left(I+\begin{bmatrix}
    0 & A^\pi CD^{e} \\
             D^\pi BA^{e}& 0
 \end{bmatrix}P_1^\#\right)(P_1^\#)^{i+1} \nonumber\\
&+&\left(\begin{bmatrix}
             0 & (CBA^\pi)^dCD^\pi \\
             BA^\pi (CBA^\pi)^d& 0\end{bmatrix}\right. \nonumber\\
&+&\left.\sum_{n=1}^{t-1}\begin{bmatrix}
             0 & (CBA^\pi)^dCD^\pi \\
             BA^\pi (CBA^\pi)^d& 0\end{bmatrix}^{n+1}\begin{bmatrix}
    A^nA^\pi & 0 \\
             0 & D^nD^\pi
 \end{bmatrix}\right) \nonumber \\
&\times&\left(\begin{bmatrix}
    A^\pi & 0\\
             0& D^\pi
 \end{bmatrix}-\begin{bmatrix}
    0 & A^\pi CD^{e} \\
             D^\pi BA^{e}& 0
 \end{bmatrix}P_1^\#\right), \label{Md}
\end{eqnarray}
where $t=\max\{\ind(A),\ind(D)\}$, $r=\ind\left(\begin{bmatrix}
    AA^\pi & CD^\pi \\
             BA^\pi & DD^\pi
 \end{bmatrix}\right)$, and
$$P_1^\#=
\begin{bmatrix}
   A^d+A^d C z^\#  B A^d&-A^d C z^\# \\-z^\# B A^d& z^\#
\end{bmatrix}=
\begin{bmatrix}
s^\#& -s^\# C D^d\\- D^d B s^\# & D^d B s^\#  C D^d+D^d
\end{bmatrix}.
$$
 \end{theorem}
\begin{proof} We can write $M=P+Q$, where
$$P=\begin{bmatrix}
    AA^{e} & CD^{e} \\
             BA^{e} & DD^{e}
 \end{bmatrix}\quad{\rm and}\quad Q=\begin{bmatrix}
    AA^\pi & CD^\pi \\
             BA^\pi & DD^\pi
 \end{bmatrix}.$$
Since $PQ=0$, then, by Lemma \ref{P+Q},
$$M^d=Q^{\pi}\sum_{i=0}^{\ind(Q)-1}Q^{i}(P^{d})^{i+1}+\sum_{i=0}^{\ind(P)-1}(Q^{d})^{i+1}P^{i}P^{\pi}.$$
Suppose that $P=P_1+P_2$, where
$$P_1=\begin{bmatrix}
    AA^{e} & A^{e}CD^{e} \\
             D^{e}BA^{e} & DD^{e}
 \end{bmatrix}\quad{\rm and}\quad P_2=\begin{bmatrix}
    0 & A^\pi CD^{e} \\
             D^\pi BA^{e}& 0
 \end{bmatrix}.$$ Then $P_1P_2=0$ and $P_2^2=0$. Using Lemma \ref{P+Q}, we get
$$P^d=P_1^\#+P_2(P_1^\#)^2=(I+P_2P_1^\#)P_1^\#,$$ where the expression for $P_1^\#$ follows by Theorem \ref{equivalent}.
By $P_1P_1^\#=\begin{bmatrix}
    A^{e} & 0 \\
 0 & D^{e}
 \end{bmatrix}$, we observe that, for \ $i=1,2,\dots$,
$$P^iP^\pi=\begin{bmatrix}
    AA^{e} & CD^{e} \\
             BA^{e} & DD^{e}
 \end{bmatrix}^{i}\left(\begin{bmatrix}
    A^\pi & 0\\
             0& D^\pi
 \end{bmatrix}-\begin{bmatrix}
    0 & A^\pi CD^{e} \\
             D^\pi BA^{e}& 0
 \end{bmatrix}P_1^\#\right)=0.$$
For $$Q_1=\begin{bmatrix}
    AA^\pi & 0 \\
             0 & DD^\pi
 \end{bmatrix}\quad{\rm and}\quad Q_2=\begin{bmatrix}
    0 & CD^\pi \\
             BA^\pi & 0
 \end{bmatrix},$$ we have $Q=Q_1+Q_2$, $Q_1Q_2=0$ and $Q_1$ is nilpotent. By \cite[Theorem 2.1]{COD},
$$Q_2^{d}=\begin{bmatrix}
             0 & (CBA^\pi)^dCD^\pi \\
             BA^\pi (CBA^\pi)^d& 0\end{bmatrix}.$$ Applying  Lemma \ref{P+Q} again, we obtain
$$Q^d=\sum_{n=0}^{\ind(Q_1)-1}(Q_2^{d})^{n+1}Q_1^{n}$$ and so
$$Q^\pi=Q_2^\pi-\sum_{m=1}^{\ind(Q_1)-1}(Q_2^d)^mQ_1^m.$$
By direct computation, we obtain \eqref{Md}.
\end{proof}

As a consequence of Theorem \ref{applMd-1}, we obtain another expressions for $M^d$.

\begin{corollary}\label{cor43} Let $t$, $r$ and $P_1^\#$ be defined as in Theorem \ref{applMd-1}.

(a) If $AC=0$ and $DB=0$, then
\begin{eqnarray}M^d&=&\left(\begin{bmatrix}
    (CBA^\pi)^\pi & 0 \\
             0 & (BCD^\pi)^\pi
 \end{bmatrix}\right. \nonumber\\
 &-&\left.\sum_{m=1}^{t-1}\begin{bmatrix}
             0 & (CBA^\pi)^dCD^\pi \\
             B(CBA^\pi)^d& 0\end{bmatrix}^m\begin{bmatrix}
    A^mA^\pi & 0 \\
             0 & D^mD^\pi
 \end{bmatrix}\right)\nonumber\\
&\times&\sum_{i=0}^{r-1}\begin{bmatrix}
    AA^\pi & CD^\pi \\
             BA^\pi & DD^\pi
 \end{bmatrix}^{i}\begin{bmatrix}
    (A^d)^{i+1} & C(D^d)^{i+2} \\
            B(A^d)^{i+2}& (D^d)^{i+1}
 \end{bmatrix} \nonumber \\
&+&\begin{bmatrix}
             -(CBA^\pi)^dCBA^d & (CBA^\pi)^dCD^\pi \\
             B(CBA^\pi)^d& -B(CBA^\pi)^dCD^d\end{bmatrix} \nonumber \\
&+&\sum_{n=1}^{t-1}\begin{bmatrix}
             0 & (CBA^\pi)^dCD^\pi \\
             B(CBA^\pi)^d& 0\end{bmatrix}^{n+1}\begin{bmatrix}
    A^nA^\pi & 0 \\
             0 & D^nD^\pi \label{Md2}
 \end{bmatrix}.
\end{eqnarray}

(b) If any of (1)--(12) in Theorem \ref{equivalent} is valid, $CD^\pi=0$ and $BA^\pi=0$, then
\begin{equation}M^d=\sum_{i=0}^{t-1}\begin{bmatrix}
    AA^\pi & 0 \\
             0 & DD^\pi
 \end{bmatrix}^{i}\left(I+\begin{bmatrix}
    0 & A^\pi C \\
             D^\pi B& 0
 \end{bmatrix}P_1^\#\right)(P_1^\#)^{i+1}. \label{Md3}
\end{equation}
 \end{corollary}

If any one of equivalent statements in Theorem
\ref{equivalent} holds, $A^\pi CD=0$, and $D^\pi BA=0$, then one more
formula for $M^d$ is given.

\begin{theorem} \label{applMd-2}
 If any one of (1)--(12) in Theorem \ref{equivalent} is valid, $A^\pi CD=0$, and $D^\pi BA=0$, then
\begin{eqnarray}M^d&=&\left(\begin{bmatrix}
    A^\pi & 0\\
             0& D^\pi
 \end{bmatrix}-P_1^\#\begin{bmatrix}
    0 & A^eCD^\pi\\
             D^eBA^\pi& 0
 \end{bmatrix}\right) \nonumber\\
 &\times&
 \left[\begin{bmatrix}
             0 & (A^\pi CB)^dA^\pi C\\
             D^\pi B(A^\pi CB)^d& 0\end{bmatrix}\right.\nonumber\\
&+&\left.\sum_{n=1}^{t-1}\begin{bmatrix}
    A^nA^\pi & 0 \\
             0 & D^nD^\pi
 \end{bmatrix}\begin{bmatrix}
             0 & (A^\pi CB)^dA^\pi C\\
             D^\pi B(A^\pi CB)^d\end{bmatrix}^{n+1}\right] \nonumber\\
 &+&\sum_{i=0}^{q-1}(P_1^\#)^{i+1}\left(I+P_1^\#\begin{bmatrix}
    0 & A^e CD^\pi \\
             D^e BA^\pi& 0
 \end{bmatrix}\right)\begin{bmatrix}
    AA^\pi & A^\pi C \\
             D^\pi B & DD^\pi
 \end{bmatrix}^{i} \nonumber\\
 &\times&
 \left[\begin{bmatrix}
    (A^\pi CB)^\pi & 0 \\
             0 & (D^\pi BC)^\pi
 \end{bmatrix}\right. \nonumber\\
 &-&\left.\sum_{m=1}^{t-1}\begin{bmatrix}
    A^mA^\pi & 0 \\
             0 & D^mD^\pi
 \end{bmatrix}\begin{bmatrix}
             0 & (A^\pi CB)^dA^\pi C\\
             D^\pi B(A^\pi CB)^d\end{bmatrix}^m\right],\label{Md4}
\end{eqnarray}
where $t=\max\{\ind(A),\ind(D)\}$, $q=\ind\left(\begin{bmatrix}
    AA^\pi & A^\pi C \\
             D^\pi B & DD^\pi
 \end{bmatrix}\right)$ and
$P_1^\#$ is defined as in Theorem \ref{applMd-1}.
 \end{theorem}

\begin{proof} Using the following decompositions
$$M=\begin{bmatrix}
    AA^{e} & A^{e}C \\
             D^{e}B & DD^{e}
 \end{bmatrix}+\begin{bmatrix}
    AA^\pi & A^\pi C \\
             D^\pi B & DD^\pi
 \end{bmatrix}:=P+Q,$$
$$P=\begin{bmatrix}
    AA^{e} & A^{e}CD^{e} \\
             D^{e}BA^{e} & DD^{e}
 \end{bmatrix}+\begin{bmatrix}
   0 & A^{e}CD^\pi \\
             D^{e}BA^\pi & 0
 \end{bmatrix}:=P_1+P_2,$$
$$Q=\begin{bmatrix}
    AA^\pi & 0 \\
             0 & DD^\pi
 \end{bmatrix}+\begin{bmatrix}
    0 & A^\pi C \\
             D^\pi B & 0
 \end{bmatrix}:=Q_1+Q_2,$$
we establish \eqref{Md4} similar to   the proof of Theorem \ref{applMd-1}.
\end{proof}

Applying Theorem \ref{applMd-2}, we obtain the following result.

\begin{corollary} Let $t$, $q$ and $P_1^\#$ be defined as in Theorem \ref{applMd-2}.

(a) If $CD=0$ and $BA=0$, then
\begin{eqnarray*}M^d&=&
 \begin{bmatrix}
             -A^dCB(A^\pi CB)^d & (A^\pi CB)^dA^\pi C\\
             D^\pi B(A^\pi CB)^d& -D^dB(A^\pi CB)^dC\end{bmatrix}\\
&+&\sum_{n=1}^{t-1}\begin{bmatrix}
    A^nA^\pi & 0 \\
             0 & D^nD^\pi
 \end{bmatrix}\begin{bmatrix}
             0 & (A^\pi CB)^dA^\pi C\\
             D^\pi B(A^\pi CB)^d\end{bmatrix}^{n+1}\\
 &+&\sum_{i=0}^{q-1}\begin{bmatrix}
    (A^d)^{i+1} & (A^d)^{i+2}C \\
            (D^d)^{i+2}B& (D^d)^{i+1}
 \end{bmatrix}\begin{bmatrix}
    AA^\pi & A^\pi C \\
             D^\pi B & DD^\pi
 \end{bmatrix}^{i}\\
 &\times&
 \left[\begin{bmatrix}
    (A^\pi CB)^\pi & 0 \\
             0 & (D^\pi BC)^\pi
 \end{bmatrix}\right.\\
 &-&\left.\sum_{m=1}^{t-1}\begin{bmatrix}
    A^mA^\pi & 0 \\
             0 & D^mD^\pi
 \end{bmatrix}\begin{bmatrix}
             0 & (A^\pi CB)^d C\\
             D^\pi B(A^\pi CB)^d\end{bmatrix}^m\right].
\end{eqnarray*}

(b) If any of (1)--(12) in Theorem \ref{equivalent} is valid, $A^\pi C=0$, and $D^\pi B=0$, then
\begin{eqnarray*}M^d&=&\sum_{i=0}^{t-1}(P_1^\#)^{i+1}\left(I+P_1^\#\begin{bmatrix}
    0 & CD^\pi \\
             BA^\pi& 0
 \end{bmatrix}\right)\begin{bmatrix}
    AA^\pi & 0\\
             0 & DD^\pi
 \end{bmatrix}^{i}.
\end{eqnarray*}
 \end{corollary}

We now provide an example  to illustrate our
results.

\begin{example} Consider the $2\times 2$ block matrices $M=\sbmatrix{cc}
A&C\\B&D\endsbmatrix$, where $A=\sbmatrix{ccc} 1&1&0\\0&0&0\\0&0&0
\endsbmatrix$, $C=\sbmatrix{cc}
0&0\\0&0\\0&2\endsbmatrix$, $B=\sbmatrix{ccc} 1&1&1\\0&0&0
\endsbmatrix$, and $D=\sbmatrix{cc} 0&1\\0&0
\endsbmatrix$. Then $A^\#=A$ and $D^d=0$. Since $AC=0$ and $DB=0$, we
apply Corollary \ref{cor43}(a) to obtain
$$M^d=\sbmatrix{ccccc} 1&1&0&0&0\\0&0&0&0&0\\0&0&0&0&0\\1&1&0&0&0\\0&0&0&0&0
\endsbmatrix.$$
\end{example}


\bibliographystyle{amsplain}

\end{document}